\documentclass{article}
[12pt , a4paper]
\usepackage[T1]{fontenc}
\usepackage[english]{babel}
\usepackage{amsmath}
\usepackage{amssymb,amsthm}
\usepackage{mathrsfs}
\usepackage{marvosym}
\usepackage{amsfonts}
\usepackage{yfonts}
\usepackage{graphicx}
\usepackage{eurosym}
\usepackage{enumerate}
\usepackage{makeidx}
\usepackage{multicol}
\usepackage{stmaryrd}
\usepackage[all]{xy}
\usepackage{fancyhdr}
\usepackage[pdftex]{hyperref}
\usepackage{color}
\DeclareGraphicsExtensions{.jpg,.pdf,.mps,.png}
\definecolor{R}{rgb}{0.5, 0, 0}
\definecolor{B}{rgb}{0, 0, 0.5}
\usepackage[margin=3cm]{geometry}
\usepackage[bottom]{footmisc}

\usepackage[small,hang]{caption}

\newtheoremstyle{note}
{4 mm}
{1 mm}
{\itshape}
{}
{\bf\sffamily}
{.}
{.5em}
{}

\theoremstyle{note}
\newtheorem{thm}{Theorem}

\newtheorem{cor}[thm]{Corollary}

\newtheoremstyle{note2}
{4 mm}
{1 mm}
{}
{}
{\bf\sffamily}
{.}
{.5em}
{}

\theoremstyle{note2}

\DeclareMathOperator{\Id}{Id}

\DeclareMathOperator{\Ad}{Ad}
\DeclareMathOperator{\ad}{ad}

\DeclareMathOperator{\fa}{\mathfrak{a}}

\DeclareMathOperator{\fg}{\mathfrak{g}}

\DeclareMathOperator{\fk}{\mathfrak{k}}

\DeclareMathOperator{\fm}{\mathfrak{m}}

\DeclareMathOperator{\fp}{\mathfrak{p}}

\title{\textbf{Notes on a Lie algebraic relation}} \date{November 2016} \author{{\Large \textsc{St\'ephane Korvers}} \\ \hspace{1 mm} \vspace{-3 mm} \\ \small{Universit\'e du Luxembourg} \\ \hspace{1 mm} \vspace{-7 mm} \\ \small{FSTC, Unit\'e de Recherche en Math\'ematiques} \\ \hspace{1 mm} \vspace{-7 mm} \\ \small{rue Richard Coudenhove-Kalergi, 6} \\ \hspace{1 mm} \vspace{-7 mm} \\ \small{L-1359 Luxembourg, Grand Duchy of Luxembourg} \\ \hspace{1 mm} \vspace{-7 mm} \\ {\small \textit{E-mail:} korvers.s@gmail.com} \\ \hspace{1 mm} \vspace{-3 mm} \\ \small \texttt{Research supported by } \\ \hspace{1 mm} \vspace{-7.5 mm} \\ \small \texttt{the \emph{Fonds National de la Recherche},} \\ \hspace{1 mm} \vspace{-7.5 mm} \\ \small \texttt{\emph{FNR/AFR-Postdoc grant} 8960322.}} 

\begin{document}


\renewcommand{\proofname}{\textbf{Proof.}}
\renewcommand{\qedsymbol}{\Squaresteel}
\renewcommand{\labelitemi}{$\bullet$}

\pagestyle{fancy} 
\fancyhf{} 
\cfoot{\thepage}
\lhead{Notes on a Lie algebraic relation}
\rhead{S. Korvers}

\thispagestyle{empty}

\maketitle

\begin{center} \small
\textbf{Abstract}

\textit{We exhibit explicit orthogonal decompositions of every multidimensional restricted root space of a real semi-simple Lie algebra. Although this interesting relation is quite elementary, we did not find it in the literature. \\We then show a link between this result and a radiality property of smooth functions \\on $G$-homogeneous spaces when $G$ is a real semi-simple Lie group.}
\end{center} \normalsize

\vspace{3 mm}

General framework results on real semi-simple Lie algebras that are mentioned below are standard and can be found e.g. in the reference \cite[Chapter 6]{Kn02}.

Let $\fg$ be a real semi-simple Lie algebra and $\sigma$ a Cartan involution of $\fg$. We denote by $\fg = \fk \oplus \fp$ the corresponding Cartan decomposition with $\sigma = \Id_{\fk} \oplus -\Id_{\fp}$ and $\beta$ the Killing form of $\fg$. We then have $\beta\left(X, Y\right) = 0$ for each $X \in \fk$ and $Y \in \fp$. Let's consider $\fa$ an abelian Lie subalgebra of $\fg$ contained in $\fp$ and maximal for this property. For each linear form $\left[\lambda : \mathfrak{a} \rightarrow \mathbb{R}\right] \in \mathfrak{a}^\star$, we set $$\mathfrak{g}_\lambda := \left\{X \in \mathfrak{g} \,:\, \left[H, X\right] = \lambda\left(H\right) X \text{ for each } H \in \mathfrak{a}\right\} \subset \mathfrak{g}.$$ 
A linear form $\lambda \in \mathfrak{a}^\star \backslash \left\{0\right\}$ such that $\mathfrak{g}_\lambda$ is non trivial is called \emph{(restricted) root of $\fg$}. The set of all these roots is denoted by $\Sigma \subset \mathfrak{a}^\star$. For $\lambda \in \Sigma$, the subspace $\fg_\lambda$ is called \emph{(restricted) root space} of $\fg$.  
It is well known that the Lie algebra $\fg$ admits the \emph{root space decomposition} $$\mathfrak{g} \,=\, \mathfrak{g}_0 \,\oplus\, \left(\bigoplus_{\lambda \in \Sigma} \mathfrak{g}_\lambda\right)$$
and we have $\left[\mathfrak{g}_\lambda, \mathfrak{g}_\mu\right] \subset \mathfrak{g}_{\lambda + \mu}$ and $\mathfrak{g}_{-\lambda} = \sigma\left(\mathfrak{g}_\lambda\right)$ for each $\lambda, \mu \in \mathfrak{a}^\star$. In addition, the subspace $\fg_0$ is a Lie subalgebra of $\fg$ which admits a decomposition $$\fg_0 = \fa \oplus \fm \text{ \,\,with\,\,\, } \mathfrak{m} := \left\{X \in \fk : \left[H, X\right] = 0 \text{ for each } H \in \mathfrak{a}\right\}.$$
As $\beta$ is positive definite on $\fa \times \fa$, for $\lambda \in \fa^\star$, we can introduce $H_\lambda \in \fa$ as the unique element in $\fa$ such that $\beta\left(H_\lambda, H\right) = \lambda\left(H\right)$ for each $H \in \fa$. The set $\left\{H_\lambda \,:\, \lambda \in \Sigma\right\}$ spans $\fa$ and we have 
\begin{eqnarray}\label{1}
\left[X, \sigma\left(X\right)\right] = \beta\left(X, \sigma\left(X\right)\right) H_\lambda
\end{eqnarray}
for all $\lambda \in \Sigma$ and $X \in \fg_\lambda$. As a consequence, it is easy to remark that 
\begin{eqnarray}\label{1b}
\left[X, \sigma\left(Y\right)\right]_{\fm} := \left[X, \sigma\left(Y\right)\right] - \beta\left(X, \sigma\left(Y\right)\right) H_\lambda \,\in\, \fm
\end{eqnarray}
for each $\lambda \in \Sigma$ and $X, Y \in \fg_\lambda$. In fact, as we have $\left[X, \sigma\left(Y\right)\right]_{\fm} \in \fg_0 = \fa \oplus \fm$ where $\fa \subset \fp$ and $\fm \subset \fk$ orthogonal with respect to $\beta$, showing this assertion is equivalent to showing that $\beta\left(\left[X, \sigma\left(Y\right)\right]_{\fm}, H\right) = 0$ for each $H \in \fa$. This follows from the properties of the root space decomposition of $\fg$ and $\ad$-invariance of $\beta$ through the equalities
\begin{eqnarray}\nonumber
\beta\left(\left[X, \sigma\left(Y\right)\right] - \beta\left(X, \sigma\left(Y\right)\right) H_\lambda, H_\mu\right) &=& \beta\left(X, \left[\sigma\left(Y\right), H_\mu\right]\right) - \beta\left(X, \sigma\left(Y\right)\right) \beta\left(H_\lambda, H_\mu\right) \\ \nonumber &=& \beta\left(X, -\left(-\lambda\right)\left(H_\mu\right) \sigma\left(Y\right)\right) - \beta\left(X, \sigma\left(Y\right)\right) \lambda\left(H_\mu\right) \,\,\,=\,\,\, 0,
\end{eqnarray}
for an arbitrary root $\mu \in \Sigma$.

\noindent As the symmetric bilinear form $\beta^\sigma : \left(X, Y\right) \in \fg \times \fg \mapsto \beta^\sigma\left(X, Y\right) := - \beta\left(X, \sigma\left(Y\right)\right)$ is positive definite, the pair $\left(\fg_\lambda, \beta^\sigma\right)$ defines a Euclidian vector space for each $\lambda \in \Sigma$. We can now state our result. 

\begin{thm}\label{1}
Let's consider $\lambda \in \Sigma$ and $X \in \fg_\lambda\backslash\left\{0\right\}$. Then, we have $$\left[\fm, X\right] = X^{\perp\left(\lambda\right)} := \left\{Y \in \fg_\lambda : \beta^\sigma\left(X, Y\right) = 0\right\}.$$ In particular, the root space $\fg_\lambda$ admits the decomposition $\fg_\lambda = \mathbb{R} X \oplus \left[\fm, X\right]$.
\end{thm} 

\begin{proof}
For each $Y \in \fm$, by using successively the $\ad$-invariance of $\beta$, relation (\ref{1}) and the orthogonality of $\fa \subset \fp$ and $\fm \subset \fk$, we get
\begin{eqnarray}\nonumber
\beta\left(X, \sigma\left(\left[X, Y\right]\right)\right) = \beta\left(\left[X, \sigma\left(X\right)\right], Y\right) = \beta\left(X, \sigma\left(X\right)\right) \beta\left(H_\lambda, Y\right) = 0.
\end{eqnarray}
Therefore, the inclusion $\left[\fm, X\right] \subseteq X^{\perp\left(\lambda\right)}$ is immediate. In order to show the reverse inclusion, let's consider an arbitrary element $X^\prime \in X^{\perp\left(\lambda\right)}$. From relation (\ref{1}) and the Jacobi identity, we get \begin{eqnarray} \label{p13} X^\prime = \left(\frac{1}{\lambda\left(H_\lambda\right)}\right) \left[H_\lambda, X^\prime\right] = \left(\frac{1}{\lambda\left(H_\lambda\right) \beta\left(X, \sigma\left(X\right)\right)}\right) \left(\left[\left[X, X^\prime\right], \sigma\left(X\right)\right] + \left[X, \left[\sigma\left(X\right), X^\prime\right]\right]\right).\end{eqnarray}
We set $X^{\prime\prime} := \left[\left[X, X^\prime\right], \sigma\left(X\right)\right]$. It is clear that $X^{\prime\prime} \in X^{\perp\left(\lambda\right)}$. In particular, if we iterate the development above with $X^{\prime\prime}$, we have
\begin{eqnarray}\label{p11} X^\prime = \left(\frac{1}{\lambda\left(H_\lambda\right) \beta\left(X, \sigma\left(X\right)\right)}\right) \left(\left(\frac{\left[\left[X, X^{\prime\prime}\right], \sigma\left(X\right)\right] + \left[X, \left[\sigma\left(X\right), X^{\prime\prime}\right]\right]}{\lambda\left(H_\lambda\right) \beta\left(X, \sigma\left(X\right)\right)}\right) + \left[X, \left[\sigma\left(X\right), X^\prime\right]\right]\right).\end{eqnarray}
Again, from relation (\ref{1}) and the Jacobi identity, we can deduce \begin{eqnarray} \label{p12}[X, X^{\prime\prime}] = \left[\left[X, X^\prime\right], \left[X, \sigma\left(X\right)\right]\right] = - 2 \lambda\left(H_\lambda\right) \beta\left(X, \sigma\left(X\right)\right) \left[X, X^\prime\right]\end{eqnarray} given that $3 \lambda \notin \Sigma$. As a consequence, by combining (\ref{p11}) with (\ref{p12}), we obtain
\begin{eqnarray} \label{p14} X^\prime = \left(\frac{1}{\lambda\left(H_\lambda\right) \beta\left(X, \sigma\left(X\right)\right)}\right) \left(\left(-2 X^{\prime\prime} + \frac{\left[X, \left[\sigma\left(X\right), X^{\prime\prime}\right]\right]}{\lambda\left(H_\lambda\right) \beta\left(X, \sigma\left(X\right)\right)}\right) + \left[X, \left[\sigma\left(X\right), X^\prime\right]\right]\right).\end{eqnarray}
If we compare (\ref{p13}) and (\ref{p14}), we get easily $3 \lambda\left(H_\lambda\right) \beta\left(X, \sigma\left(X\right)\right) X^{\prime\prime} = \left[X, \left[\sigma\left(X\right), X^{\prime\prime}\right]\right]$, and then $$X^\prime = \left(\frac{1}{\lambda\left(H_\lambda\right) \beta\left(X, \sigma\left(X\right)\right)}\right) \left[X, \left[\sigma\left(X\right), Z\right]\right] \text{\,\, where \,\,} Z := \left(\frac{1}{3 \lambda\left(H_\lambda\right) \beta\left(X, \sigma\left(X\right)\right)}\right) X^{\prime\prime} + X^\prime \in X^{\perp\left(\lambda\right)}.$$
Given that $\beta\left(Z, \sigma\left(X\right)\right) = 0$, it follows from (\ref{1b}) that $\left[\sigma\left(X\right), Z\right] \in \fm$ and the proof is complete.
\end{proof}

\noindent The following corollary can be deduced trivially from the previous theorem.

\begin{cor}
\hspace{1 mm} \newline $(a)$ \,\, If the root $\lambda \in \Sigma$ is such that $\dim\left(\fg_\lambda\right) = 1$, then $\left[\fm, X\right] = 0$. \newline $(b)$ \,\, If the Lie algebra $\fg$ is such that $\fm = 0$, all the root spaces of $\fg$ are one-dimensional.
\end{cor}

Let's consider $G$ a connected Lie groups with Lie algebra $\fg$ and $K \subset G$ a Lie subgroup of $G$ with Lie algebra $\fk$. We denote by $\mathcal{C}^\infty\left(G/K\right)$ the set of real valued smooth function on the connected homogeneous space $G/K$. It is well known that the Lie group $G$ acts differentially and transitively on $G/K$ via the action $G \times G/K \rightarrow G/K : \left(g, g^\prime K\right) \mapsto \left(gg^\prime\right) K$. In what follows, for $X \in \fg$, the notation $X^\star \in \Gamma\left(T\left(G/K\right)\right)$ will refers to the \emph{fundamental vector field} associated with $X$ which is defined at point $gK \in G/K$ by $$X^\star_{gK} := \left.\frac{d}{dt}\right|_{t=0} \left(\exp\left(- t X\right)g\right) K.$$

\begin{cor}
Let's consider $f \in \mathcal{C}^\infty\left(G/K\right)$ such that $X^\star f = 0$ for all $X \in \fm$. Then, for each $\lambda \in \Sigma$, the function $$f_\lambda : \fg_\lambda \rightarrow \mathbb{R} : X \mapsto f\left(\exp\left(X\right) K\right)$$ depends only on $\beta^\sigma\left(X, X\right)$. In particular, the function $f_\lambda$ is radial in the Euclidian vector space $\left(\fg_\lambda, \beta^\sigma\right)$.
\end{cor}

\begin{proof}
Let's fix $\lambda \in \Sigma$. For each $X \in \fm \subset \fk$ and $Y \in \fg_\lambda$, given that $\exp\left(X\right) \in K$, we have 
\begin{eqnarray} \nonumber 0 \,\,\,\, = \,\,\,\, \left(X^\star f\right)\left(\exp\left(Y\right) K\right) &=& \left.\frac{d}{dt}\right|_{t=0} f\left(\left(\exp\left(- t X\right) \exp\left(Y\right) \exp\left(t X\right)\right) K\right) \\ \nonumber &=& \left.\frac{d}{dt}\right|_{t=0} f_\lambda \left(\Ad_{\exp\left(- t X\right)}\left(Y\right)\right) \\ \nonumber &=& {{f_\lambda}_\star}_Y \left(\left[Y, X\right]\right).
\end{eqnarray}
As a consequence, in view of theorem \ref{1}, the map ${{f_\lambda}_\star}_Y$ vanishes on $Y^{\perp\left(\lambda\right)}$ for each $Y \in \fg_\lambda$ and the result follows.
\end{proof}

\end{document}